\documentclass[letterpaper,10pt,conference,twoside]{IEEEconf}  

\IEEEoverridecommandlockouts                              

\usepackage{amsmath,amssymb,color,cite}
\usepackage{graphicx}
\usepackage{subfigure}

\usepackage{latexsym}
\usepackage{algorithm}
\usepackage{algpseudocode}
\usepackage{verbatim}
\usepackage{cite}
\usepackage[english]{babel}

\usepackage{amsmath,amsthm}

\renewcommand{\theenumi}{(\roman{enumi})}

\usepackage{pgf}
\usepackage{pgfplots,tikz}
\usetikzlibrary{tikzmark}
\usepackage[font=footnotesize]{caption}
\usepackage{arydshln}

\usepackage{amsfonts}
\usepackage{amssymb}
\usepackage{amsbsy}
\usepackage{bm}

\renewcommand{\algorithmicrequire}{\textbf{Input: }}
\renewcommand{\algorithmicensure}{\textbf{Output: }}

\DeclareMathOperator{\im}{im}

\newcommand\oprocendsymbol{\hbox{$\bullet$}}
\newcommand\oprocend{\relax\ifmmode\else\unskip\hfill\fi\oprocendsymbol}

\let\leq\leqslant
\let\geq\geqslant
\let\emptyset\varnothing

\newcommand{\R}{\mathbb R}
\newcommand{\N}{\mathbb N}
\newcommand{\bT}{\mathbb T}

\newcommand{\calP}{\ensuremath{\mathcal{P}}}

\newtheorem{theorem}{Theorem}[section]

\newtheorem{lemma}[theorem]{Lemma}
\newtheorem{corollary}[theorem]{Corollary}

\theoremstyle{remark}
\newtheorem{remark}[theorem]{Remark}

\theoremstyle{definition}
\newtheorem{assumption}{Assumption}
\newtheorem{definition}[theorem]{Definition}

\newtheorem{problem}{Problem}

\makeatletter
\algnewcommand{\LineComment}[1]{\Statex \hfill \(\triangleright\) #1}
\makeatother

\renewcommand{\baselinestretch}{0.99}
\allowdisplaybreaks

\title{Using data informativity for online stabilization of unknown
  switched linear systems\thanks{This work was partially supported by
    AFOSR Award FA9550-19-1-0235.}}

\author{Jaap Eising${}^*$ \quad Shenyu Liu${}^*$ \quad Sonia
	Mart\'{\i}nez \quad Jorge Cort\'{e}s\thanks{${}^*$ Both authors
		contributed equally. Jaap Eising, Sonia
		Mart\'{\i}nez and Jorge Cort\'{e}s are with the Department of
		Mechanical and Aerospace Engineering, University of California,
		San Diego, \texttt{\{jeising,soniamd,cortes\}@ucsd.edu}. Shenyu Liu is with the School of Automation, Beijing Institute of Technology, China, \texttt{shenyuliu@bit.edu.cn}.} }

\begin{document}

\maketitle

\begin{abstract}
This work studies data-driven switched controller design for discrete-time switched linear systems. Instead of having access to the full system dynamics, an initialization phase is performed, during which noiseless measurements of the state and the input are collected for each mode. Under certain conditions on these measurements, we develop a stabilizing switched controller for the switched system. To be precise, the controller switches between identifying the active mode of the system and applying a predetermined stabilizing feedback. We prove that if the system switches according to certain specifications, this controller stabilizes the closed-loop system. Simulations on a network example  illustrate our approach.
\end{abstract}

\section{Introduction}

A switched system is a dynamical system that consists of several
modes, or subsystems. A logical rule, called the switching signal,
governs the switching between these modes~\cite{DL:03}. Because such
systems have been shown to model many applications, the study of
switched systems has attracted a lot of research interest in the
previous decades. One of the main topics of interest is the
stabilization of switched systems~\cite{HL-PJA:09}.  Known controller
designs include system matrix-based
methods~\cite{DC-LG-YL-YW:05,YS:11switched}, methods based on common
Lyapunov functions~\cite{MH-FF-GD:13,DZ-LA-JD-QZ:18} or multiple
Lyapunov functions~\cite{MSB:98,JG-PC:06,PC-JG-AA:08}.  However, all
the aforementioned design approaches are model-based, that is, they
require knowledge of the precise system dynamics in order to stabilize
the switched system. In practice, this assumption is often quite
restrictive, since the systems can be too complex to model or
uncertainties make precise modeling impossible. The problem of
controlling uncertain switched systems is studied in
\cite{HL-PJA:07,LIA-US:13}, where some parameters of the system
dynamics are assumed to be unknown but within a given range.

Uncertainty on models can be reduced by employing measurements and
constructing data-driven controls. This has become an active area of
research, leading to various papers such
as~\cite{CDP-PT:20,HJVW-JE-HLT-MKC:20}, of special relevance to this
work. Data-driven controller design for switched systems is more
challenging and has only been recently considered; see for
instance~\cite{CZ-MG-JZ:19,AK:20arXiv}. The type of systems studied
here have no external inputs; and the only controlling element is the
switching signal. On the other hand, the works
\cite{TD-MS:18,ZW-GOB-RMJ:21arXiv} focus on finding feedback control
laws that stabilize a switched linear system under arbitrary and
unknown switching signals. In this case, the desired feedback control
needs to uniformly stabilize all the modes simultaneously. Therefore,
such a controller may not exist in general and the corresponding
algorithms will be necessarily restrictive. Another relevant work is
\cite{MR-CDP-PT:22}, which proposes an online data-driven
feedback control. In this way, a stabilizing feedback control is found
based on measurements of the currently active mode of the
system. Naturally, only switched systems for which the system switches
infrequently can be stabilized.


In this paper, we study a stabilizing control design problem for
switched linear systems. In particular, we consider a situation in
which we do not have access to a model of the separate modes, nor the
precise switching signal of the system. Clearly, in order to be able
to design such a controller, it is necessary for each of the modes to
be stabilizable separately. As such, to compensate for the fact that
the dynamics of the modes are unknown, we assume that an
\textit{initialization phase} is performed. In this phase,
measurements of the state and the input are collected on a finite time
interval for each of these modes. In this sense, we have
\textit{partial information} on the modes of the system. We will
employ the methods of the informativity framework (see
e.g. \cite{HJVW-JE-HLT-MKC:20}) to develop necessary and sufficient
conditions on these measurements that guarantee that each mode is
stabilizable with a given decay rate. In particular, we note that
these measurements are not necessarily informative enough to uniquely
identify a model for each mode.

After this initialization phase, we consider the problem of operating
the switched system online. To achieve a stable behavior, our proposed
controller alternates between two phases, which consist of a
\textit{mode detection phase} and a \textit{stabilization phase}. If
the controller detects a modal switch, the mode-detection phase
applies excitatory inputs to measure system outputs and identify the
active mode. Thus, a main problem we consider is the characterization
of necessary and sufficient conditions on the \textit{online}
measurements that uniquely determine the current mode of the
system. In particular, these conditions must be such that they are
guaranteed to hold after a bounded number of steps.  After determining
the active mode, our proposed controller applies a stabilizing
feedback corresponding to this mode in a \textit{stabilization
  phase}. The controller will switch back to a mode detection phase as
soon as the solution of the system does not converge fast enough
(which is quantified in terms of Lyapunov functions). This leads to
the final problem considered in this paper: Obtain conditions that
guarantee that the stabilization phase controllers compensate for the
potential destabilization that occurs in the mode-detection phase, so
that overall closed-loop system is stable.

The paper is structured as follows. First, necessary background
notions and the problem are introduced and formulated in
Section~\ref{sec:notion}. Then, the main elements of our data-driven
switched controller design are explained in detail in
Section~\ref{sec:design}, which is followed by the stability analysis
of the closed-loop system in Section~\ref{sec:stability}. Our main
results are then supported by a numerical example in
Section~\ref{sec:simulation} and finally Section~\ref{sec:conclusion}
concludes our paper. All proofs are omitted for reasons of space and
will appear elsewhere.

\section{Problem formulation}\label{sec:notion}

Consider\footnote{Throughout the paper, we use the following
  notation. We denote by $\N$ and $\R$ the sets of non-negative
  integer and real numbers, respectively. We let $\R^{n\times m}$
  denote the space of $n\times m$ real matrices. For any
  $M\in\R^{n\times m}$, $\Vert M\Vert$ denotes the standard
  2-norm. For $P\in\R^{n\times n}$, $P\succeq 0$ (resp. $P\succ 0$)
  denotes that $P$ is positive semi-definite (resp.  definite).}  a
switched linear system
\begin{equation}\label{def:sys}
  x(t+1)=A_{\sigma(t)}x(t)+B_{\sigma(t)}u(t) ,
\end{equation}
where $x\in\R^n$ is the state,
$\sigma:\N\mapsto\{1,2,\cdots,p\}=:\calP$ is the switching signal and
$u\in\R^m$ is the control.  The dimensions $n$, $m$ and the number of
modes $p$ are known. However, the dynamics of the modes are unknown;
that is, for each mode $i\in\calP$, the matrices
$A_i\in\R^{n\times n}$ and $B_i\in \R^{n\times m}$ are
unknown. Additionally, we assume that the switching signal $\sigma$ is
also unknown; that is, we do not know in advance when the switches
happen and once a switch happens, which of the modes is
active.

To offset this lack of knowledge, we have access to a finite set of
measurements of the system. Specifically, in an
\textit{initialization} phase, we assume measurements of the
individual modes are available. After this, once the system is
running, we have access to \textit{online} measurements of the
currently active mode.  We consider the situation in which we can not
necessarily identify the model of the separate modes based on the
initialization data, nor the currently active mode based on the online
measurements.  Given this setup, our goal is to solve the following
problem.

\begin{problem}[Online switched controller design]\label{problem}
  Given initialization and online measurements, design a control law
  $u:\N\mapsto\R^m$ such that the resulting interconnection
  \eqref{def:sys} is guaranteed to satisfy the following
  \emph{input-to-state stability} (ISS)-like property: there exist
  constants $c>0,\zeta\in(0,1)$ and $r\geq 0$ (depending on the input)
  such that
  \begin{equation}\label{ISS-estimate_0}
    |x(t)|\leq c \, \zeta^t|x(0)|+r ,
  \end{equation}
  for all initial states $x(0)\in\R^n$ and all time $t\in\N$.
\end{problem}

To solve Problem~\ref{problem}, we employ the following multi-pronged
approach.  Since unique models for each of the modes cannot
necessarily be determined, we employ the concept of \textit{data
  informativity} to formulate conditions under which the
initialization measurements are enough to guarantee the existence of a
stabilizing feedback controller for each of the modes. Based on this, our
switched controller operates in two phases. In the \textit{mode
  detection phase}, the controller selects inputs that allow us to
determine, within a bounded number of steps, the active mode (which,
in general, requires less measurements than fully identifying the
system). Once the mode is identified, the controller switches to the
\textit{stabilization phase}, where the controller found in the
initialization step corresponding to the mode is
applied. Section~\ref{sec:design} formally describes the controller
and Section~\ref{sec:stability} characterizes its properties.

\section{Switched controller design}\label{sec:design}
In this section, we describe the three components of the switched
controller design: the initialization step, the mode detection phase,
and the stabilization phase.
 
\subsection{Initialization step}\label{subsec:initialization}
We start by considering the problem of finding stabilizing controllers
from pre-collected measurements, for which we resort to the notion of
data informativity.  Given measurements of the state $x$ and input $u$
signals on the time interval $\{0,\ldots, T\}$, we define the
matrices:
\[
  X:=
  \begin{bmatrix}x(0) & \!\!\cdots\!\!  &
    x(T)\end{bmatrix}\!,\hspace{1em} U_- :=\begin{bmatrix} u(0) &
    \!\!\cdots\!\!  & u(T-1)\end{bmatrix}.
\]
In the following, we identify these matrices with the
measurements. For convenience, we also define
\[
  X_+ := \begin{bmatrix} x(1) & \!\!\cdots\!\! &x(T) \end{bmatrix}\!,
  \hspace{1em} X_- := \begin{bmatrix} x(0) & \!\!\cdots\!\!  &
    x(T-1) \end{bmatrix}.
\]
The set of consistent systems is given~by
\begin{equation*}
  \Sigma(U_-,X)\! := \big\{(A,B)
  :X_+ \! =\! AX_-\!+ \! BU_-\big\}\!\subseteq\R^{n\times n}\times\R^{n\times m}. 
\end{equation*}
We are interested in characterizing properties of the true system
based on the measurements. However, the set $\Sigma(U_-,X)$ might
contain more than one pair of system matrices. This means that, we can
only conclude that, for instance a feedback gain $K$ stabilizes the
true system if this gain stabilizes any system whose system matrices
are in $\Sigma(U_-,X)$. The following notion captures when the data is
sufficiently informative to do this with a uniform decay rate.

\begin{definition}[Informativity for uniform
  stabilization]\label{def:informativity for uniform stabilization}
  The data $(U_-,X)$ is \textit{informative for uniform stabilization
    by state feedback with decay rate} $\lambda\in(0,1)$ if there
  exist $K\in\R^{m\times n}$ and $P\in\R^{n\times n}, P\succ 0$ such
  that
  \begin{equation}\label{eqn:informativity for uniform stabilization}
    (A+BK)^\top P (A+BK)\preceq\lambda P \quad\forall (A,B)\in\Sigma(U_-,X).
  \end{equation}
\end{definition}

When $(U_-,X)$ satisfies this definition and the matrices $K$ and $P$
are known, one can apply the control $u=Kx$ to the system, choose the
function $x \mapsto V(x)=x^\top Px$ and conclude
\begin{align}\label{eq:decay-rate}
  V(x(t+1))
    & \leq \lambda x(t)^\top P x(t)=\lambda V(x(t)),
\end{align}
for all $t \in \N$. This means that, even though the exact dynamics of
the system are unknown, the feedback gain $K$ is stabilizing, with
Lyapunov certificate $V$ with decay rate $\lambda$. Nevertheless,
Definition~\ref{def:informativity for uniform stabilization} does not
provide a constructive way of finding $K$ and $P$
satisfying~\eqref{eqn:informativity for uniform stabilization}. The
next result formulates the problem of finding such $K$ and $P$ as a
\emph{linear matrix inequality} (LMI) problem and it also shows that
the feasibility of this inequality is a necessary and sufficient
condition for the data $(U_-,X)$ to be informative for uniform
stabilization with decay rate $\lambda$ by state feedback.

\begin{theorem}[Conditions for informativity for uniform
  stabilization]\label{thm:common_K}
  The data $(U_-,X)$ is informative for uniform stabilization by state
  feedback with decay rate $\lambda$ iff there exist
  $Q\in\R^{n\times n}, Q\succ 0$ and $L\in\R^{m\times n}$ such that
  \begin{equation}\label{eqn:common_K_LMI}
    \begin{bmatrix}
      \lambda Q& 0 & 0 &0\\0&0&0 &Q\\0&0&0&L\\0&Q&L^\top &Q
    \end{bmatrix}+\begin{bmatrix} X_+\\-X_-\\-U_-\\0
    \end{bmatrix}\begin{bmatrix}
      X_+\\-X_-\\-U_-\\0
    \end{bmatrix}^\top\succeq 0.
  \end{equation}
  Moreover, the matrices $K:=LQ^{-1}, P:=Q^{-1}$ satisfy
  \eqref{eqn:informativity for uniform stabilization}.
\end{theorem}

Theorem~\ref{thm:common_K} can be concluded using the methods of
from~\cite{HJVW-MKC-JE-HLT:22,HJW-MKC:21arXiv}.  With these notions in
place, we make the following assumption.

\begin{assumption}[Initialization phase]\label{ass:ass_on_data-I}
  Let $\lambda\in (0,1)$.  For each mode $i\in\calP$, the data
  $(U_-^i,X^i)$ is informative for uniform stabilization by state
  feedback with decay rate $\lambda$.
\end{assumption}

\subsection{Mode detection phase}\label{subsec:MD}

After the initialization step, the system is in operation and
additional \textit{online} measurements, denoted by
$(U^{\textnormal{on}}_-,X^{\textnormal{on}})$, are collected from the
system when it dwells in one mode.  We then face the problem of
determining which of the $p$ different systems has generated the
online measurements $(X^{\textnormal{on}},U_-^{\textnormal{on}})$. In
this section, we propose an algorithm to find the data-generating
mode. The notion of data compatibility plays a key role in achieving this
goal.
 
\begin{definition}[Compatibility of data]
  For $i \in \calP$, the data $(U_-^i,X^i)$ and
  $(U_-^{\textnormal{on}},X^{\textnormal{on}})$ are
  \textit{compatible} if there exists a system consistent with both,
  i.e.,
  $\Sigma(U^i_-,X^i)\cap\Sigma(U^{\textnormal{on}}_-,X^{\textnormal{on}})\neq\emptyset$.
\end{definition}

The following result characterizes this property.

\begin{lemma}[Conditions for data compatibility]\label{lem:exact dis}
  For $i \in \calP$, the data $(U^i_-,X^i)$ and
  $(U^{\textnormal{on}}_-,X^{\textnormal{on}})$ are compatible iff
  \begin{equation}\label{eq:kernel}
    \ker
    \begin{bmatrix}
      X^i_-&X^{\textnormal{on}}_-\\U^i_-&U^{\textnormal{on}}_-
    \end{bmatrix}
    \subseteq  
    \ker \begin{bmatrix}X^i_+&X^{\textnormal{on}}_+\end{bmatrix}. 
  \end{equation}
\end{lemma}

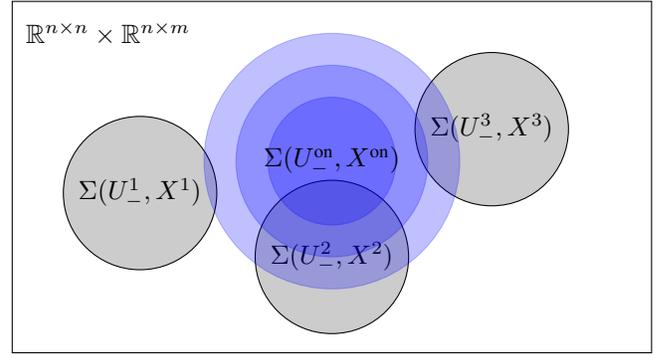
\begin{figure}
  \centering
  \begin{tikzpicture}[scale=0.85]
    \filldraw[color=black, fill=black, fill opacity=0.2] (2,3) circle (1.2) (2,3)  node [opacity=1] {$\Sigma(U^1_-,X^1)$}
    (5,2) circle (1.2) (5,2)  node [opacity=1] {$\Sigma(U^2_-,X^2)$}
    (7.5,4) circle (1.2) (7.5,4)  node [opacity=1] {$\Sigma(U^3_-,X^3)$};
    \filldraw[blue, opacity=0.25]  (5,3.5) circle (2);
    \filldraw[blue, opacity=0.25]  (5,3.5) circle (1.5);
    \filldraw[blue, opacity=0.25]  (5,3.5) circle (1) (5,3.5) node [black, opacity=1] {$\Sigma(U^{\textnormal{on}}_-,X^{\textnormal{on}})$};
    \draw (0,0.5) rectangle (10,6) (1.5,5.5) node {$\R^{n\times n}\times \R^{n\times m}$};
  \end{tikzpicture}
  \caption{Graphical interpretation of the mode detection
    scheme. Initially,
    $\Sigma(U^{\textnormal{on}}_-,X^{\textnormal{on}})$ intersects the
    sets corresponding to three different modes.  As more data is
    collected, $\Sigma(U^{\textnormal{on}}_-,X^{\textnormal{on}})$
    decreases in size (cf.  darker blue disks). When enough data is
    available, $\Sigma(U^{\textnormal{on}}_-,X^{\textnormal{on}})$
    eventually becomes compatible only with mode
    2.}\label{fig:mode_detection}
\end{figure}

For any mode detection mechanism to be successful, we require
that the initialization data are pairwise incompatible.

\begin{assumption}[Initialization phase
  --cont'd]\label{ass:ass_on_data-II}
  For each mode $i\in\calP$ the matrix pair $(A_i,B_i)$ is
  controllable. Furthermore, the data $\{(U_-^i,X^i)\}_{i\in\calP}$
  are such that $(U^i_-,X^i)$ and $(U^j_-,X^j)$ are incompatible for
  each pair $i\neq j \in\calP$.
\end{assumption}

Figure~\ref{fig:mode_detection} provides the intuition for the mode
detection scheme. Since the online measurements are generated by one
of the modes $i\in\calP$, there must exists at least one $i\in\calP$
such that $(U_-^{\textnormal{on}},X^{\textnormal{on}})$ and
$(U_-^i,X^i)$ are compatible. Since the initial data are pairwise
incompatible, when $\Sigma(U_-^{\textnormal{on}},X^{\textnormal{on}})$
is sufficiently ``small'', the mode $i$ giving $(U_-^i,X^i)$
compatible with $(U_-^{\textnormal{on}},X^{\textnormal{on}})$ must be
unique and hence the active mode is detected. This leads us to the
following definition.

\begin{definition}[Informativity for mode detection]
  The initialization data $\{(U_-^i,X^i)\}_{i\in\calP}$ and online
  data $(U_-^{\textnormal{on}},X^{\textnormal{on}})$ are
  \emph{informative for mode detection} if
  $(U_-^{\textnormal{on}},X^{\textnormal{on}})$ and $(U_-^i,X^i)$ are
  compatible for exactly one $i\in\calP$.
\end{definition}

As a result of these observations, we are interested in
\textit{generating} online data such that, after a bounded number of
steps, $\Sigma(U_-^{\textnormal{on}},X^{\textnormal{on}})$ becomes
small enough so that the data are informative for mode detection. More
precisely, given the initialization data, the problem is to find a
time horizon $T^{\textnormal{on}}$ and inputs
$u^{\textnormal{on}}(0),\ldots, u^{\textnormal{on}}(T^{\textnormal{on}}-1)$ such that
the corresponding online data
$(U_-^{\textnormal{on}},X^{\textnormal{on}})$ are informative for mode
identification.  To obtain such inputs, we adopt the experiment design
method of~\cite{HJVW:21}, which constructs inputs
$u^{\textnormal{on}}(0),\ldots, u^{\textnormal{on}}(n+m-1)$ such that
the corresponding $\Sigma(U_-^{\textnormal{on}},X^{\textnormal{on}})$
is a singleton set. Clearly, such online measurements would be
informative for mode identification. We apply these specific inputs to
obtain an upper bound to the number of steps required for the mode
detection phase. Note, however, that in general mode detection is achieved
 with less measurements than system
identification.

\begin{algorithm}[htb!]
  \caption{Mode detection algorithm per time instant}\label{alg:mode_detection}
  \algorithmicrequire $\calP,\calP_{\textrm{match}},\{U^i_-,X^i\}_{i\in\calP},U^{\textnormal{on}}_-,X^{\textnormal{on}},u_{\max}$\\
  \algorithmicensure $\calP_{\textrm{match}},U^{\textnormal{on}}_-,X^{\textnormal{on}}$
  \begin{algorithmic}[1]           
    \If{$U_-^{\textnormal{on}}\neq []$ and $x(t)\in\im X^{\textnormal{on}}-$} \Comment{Choose the next input}\label{step:mode_detection_start}
    	\State Pick $\begin{bmatrix}
    	\xi\\\eta
    	\end{bmatrix}\in\ker\begin{pmatrix}
    	X^{\textnormal{on}}_-\\U^{\textnormal{on}}_-
    	\end{pmatrix}^\top$ with $\eta\neq 0$\label{step:mode_detection_critical}
    	\State Let $u(t)\in\R^m$ be such that $|u(t)|\leq u_{\max}$ and $\xi^\top x(t)+\eta^\top u(t)\neq 0$
    	\Else  
    \State $u(t)\gets 0$
    \EndIf\label{step:mode_detection_end}
  \State Get the next state $x(t+1)$
  \State $U_-^{\textnormal{on}}\gets \begin{bmatrix}
      U_-^{\textnormal{on}}&u(t)
  \end{bmatrix}$                         \State $X^{\textnormal{on}}\gets \begin{bmatrix}
  X^{\textnormal{on}}& x(t+1)
\end{bmatrix}$ \Comment{Append the data}
  \For{$i\in\calP$}  
    \If{the inclusion \eqref{eq:kernel} is violated}      \LineComment{Data are incompatible}
    \State $\calP_{\textrm{match}}=\calP_{\textrm{match}}\backslash\{i\}$ \Comment{Eliminate mode $i$}
    \EndIf
    \EndFor 
  \end{algorithmic}
\end{algorithm}

Algorithm~\ref{alg:mode_detection} summarizes the mode detection
procedure. The strategy updates the online data
$(U_-^{\textnormal{on}},X^{\textnormal{on}})$ and a list
$\calP_{\textrm{match}}$ containing all the modes that are compatible
with the online data for each time instance. To bound the
destabilizing effect of the mode detection phase, the algorithm
modifies the experiment design method in~\cite{HJVW:21} (Steps 1 to
Steps 6) to add a parameter $u_{\max}>0$ bounding the magnitude of the
detection input.  The existence of $\eta\in\R^n$ satisfying the
conditions in Step~\ref{step:mode_detection_critical} is guaranteed
when each mode is controllable. As a consequence, it is
straightforward to check that the rank of
$\begin{pmatrix} X^{\textnormal{on}}_-\\U^{\textnormal{on}}_-
\end{pmatrix}$ increases with each step of the algorithm. As such,
after $n+m$ steps this data-matrix has full column rank, and hence
admits a unique compatible system.  As such:

\begin{corollary}\label{cor:at-most-n+m}
  Under Assumptions~\ref{ass:ass_on_data-I}
  and~\ref{ass:ass_on_data-II}, the online data
  $(U^{\textnormal{on}}_-,X^{\textnormal{on}})$ generated by
  Algorithm~\ref{alg:mode_detection} are informative for mode
  detection after at most $n+m$ time instances.
\end{corollary}

\subsection{Stabilization phase}\label{subsec:stabilization}

Given Assumption~\ref{ass:ass_on_data-I}, for each mode $i \in \calP$,
we can use the LMI~\eqref{LMI_for_MD} to find a feedback gain $K_i$
and a positive definite matrix $P_i$, such that $u=K_ix$ is a
stabilizing feedback law corresponding to the Lyapunov function given
by $V_i(x):=x^\top P_i x$.  Once the current mode has been identified,
this provides the control input for the stabilization phase.

Note, however, that the switching signal is unknown. This means that,
once the system has switched its mode, the current controller may no
longer be stabilizing the (new) closed-loop system and needs to be
updated. As such, we introduce a method of \textit{switch triggering}
to have the controller go back to the mode detection phase. For this,
we employ the decay rate guaranteed by the applied feedback gain
$K_i$. To be precise, recall from~\eqref{eq:decay-rate} that when the
controller applies the feedback corresponding with the active mode,
\begin{equation}\label{what_to_obey}
  V_i(x(t+1))\leq \lambda V_i(x(t)).
\end{equation}
This inequality provides a criterium to trigger switches: the
controller switches to the mode detection phase whenever the one-step
inequality \eqref{what_to_obey} is violated.

\begin{remark}[On switching to the mode detection
  phase]\label{rem:unsynchronized}
  Note that the system~\eqref{def:sys} may switch its mode while the
  inequality \eqref{what_to_obey} is preserved. In this case, our
  controller will not immediately change to the mode detection phase,
  causing the switching of the system and the switching of the
  controller to be unsynchronized. As we show later when analyzing the
  properties under the proposed controller, this does not affect
  stability, as the Lyapunov function $t\mapsto V_i(x(t))$ is still
  decreasing at the desired rate~$\lambda$.
  \oprocend
\end{remark}

Algorithm~\ref{alg:controller_1} summarizes the proposed
controller. After the initialization step, the algorithm starts in the
mode detection phase, indicated by the flag variable
$S_{\textrm{phase}}$. During the mode detection phase,
Algorithm~\ref{alg:mode_detection} is executed for each iteration
until the list of compatible modes contains only a single
element. After that, the controller's determination of the active mode
$\sigma_d$ is assigned with that remaining element and
$S_{\textrm{phase}}$ is toggled to $1$, indicating the switching into
stabilization phase.  During the stabilization phase. the control
$u(t)=K_{\sigma_d}x(t)$ is applied. Moreover,
when~\eqref{what_to_obey} is violated, the controller switches back to
mode detection phase by toggling $S_{\textrm{phase}}$ to $0$. In the
meantime, $X^{\textnormal{on}}, U_-^{\textnormal{on}}$ are reset to
record the new online data.

\begin{algorithm}[htb!]
  \caption{Data-driven switched feedback controller}\label{alg:controller_1}
  \algorithmicrequire $\calP,\{U^i_-,X^i,K_i,P_i\}_{i\in\calP},\lambda$
  \begin{algorithmic}[1]
  \State $\calP_{\textrm{match}}\gets\calP$
  \State $S_{\textrm{phase}}$ $\gets 0$ \Comment{Initialize to mode detection phase}
  \State $X^{\textnormal{on}}\gets x(0)$                 
  \State $U_-^{\textnormal{on}}\gets []$                 \Comment{Initialize online data}
  \While{the system \eqref{def:sys} is running}
  \If{$S_{\textrm{phase}}$ $=0$}                    \Comment{Mode detection phase} 
  \State Run Algorithm~\ref{alg:mode_detection}
  \State Update the variables $\calP_{\textrm{match}}, X^{\textnormal{on}},U^{\textnormal{on}}_-$    
  \If{$|\calP_{\textrm{match}}|=1$}
  \State $\sigma_d\in \calP$ \label{step:mode_to_switch}  \Comment{Set the mode of the controller}
  \State $\calP_{\textrm{match}}\gets\calP$
  \State $S_{\textrm{phase}}$ $\gets 1$                \Comment{Toggle phase}
  \EndIf    
  \Else   \Comment{Stabilization phase}
   \State Apply control $u(t)=K_{\sigma_d}x(t)$ 
\State Get the next state $x(t+1)$
 \If{$x(t+1)^\top P_{\sigma_d}x(t+1)> \lambda x(t)^\top  P_{\sigma_d} x(t)$}    \LineComment{Trigger for phase change}
 \State $X^{\textnormal{on}}\gets x(t)$                 
 \State $U_-^{\textnormal{on}}\gets []$                 \Comment{Reset online data}
 \State $S_{\textrm{phase}}$ $\gets 0$             
 \EndIf
    \EndIf
    \State $t\gets t+1$                     \Comment{Update the time}
    \EndWhile
  \end{algorithmic}
\end{algorithm}

We recall that ``exciting enough" inputs are used during the mode
detection phase, which may have negative effect on the stability of
the closed-loop system. Therefore, the alternation between mode
detection and stabilization in our control mechanism makes it
nontrivial to ensure the stability
guarantee~\eqref{ISS-estimate_0}. Establishing it is the subject of the
next section.  

\section{Stability analysis of the closed-loop
  system}\label{sec:stability}


In this section, we analyze the stability properties of the
closed-loop system and identify conditions so that it enjoys the
ISS-like property~\eqref{ISS-estimate_0}. Note that the switching
nature of the closed-loop system is due not only to the unknown
switching signal $\sigma$, but also to the control switches induced by
Algorithm~\ref{alg:controller_1}.  We make the following assumptions
regarding the system switching frequency.

\begin{assumption}[Assumptions on the switching
  signal]\label{ass:detection_after_exercution}
  Let $\bT^m:=\{t_1^m,t_2^m,\cdots\}$
  (resp. $\bT^s:=\{t_1^s,t_2^s,\cdots\}$) be the ordered sets
  consisting of the first time instants of each mode identification
  phase (resp.  stabilization phase). Then, the following holds:
  \begin{enumerate}
  \item\label{itm:asump 1} The system does not switch while the controller is in mode
    detection phase;
  \item\label{itm:asump adt} Let $N(t_a,t_b)$ be the total number of mode identification
    phases over the time interval $[t_a,t_b)$, that is,
    $ N(t_a,t_b):= |[t_a,t_b)\cap \bT^m|$, where $|\cdot|$ denotes
    cardinality.  There exists $\tau$ and $N_0\geq 1$ such that
    \begin{equation}\label{ADT_condition}
      N(t_a,t_b) \leq N_0+\frac{t_b-t_a}{\tau}\quad\forall t_a,t_b\in
      \N, t_a< t_b;
    \end{equation}
 
  \item\label{itm:asump aat} Let $M(t_a,t_b)$ be the total time spent in mode
    identification phases over the time interval $[t_a,t_b)$, that is,
    $M(t_a,t_b):=\sum_{t=t_a}^{t_b-1}{\mathbf 1}(t)$, where
    \[ \mathbf 1 (t):=\begin{cases} 1& \mbox{ if }t\in [t_i^m,t_i^s)
        \mbox{ for some }i\in\N,
        \\
	0&\mbox{ otherwise. }
      \end{cases}
    \] 
    There exists $\eta\in[0,1]$ and $T_0\geq 0$ such that
    \begin{equation}\label{AAT_condition}
      M(t_a,t_b)\leq T_0+\eta(t_b-t_a)\quad\forall t_a,t_b\in \N, t_a< t_b.
    \end{equation}
  \end{enumerate}
\end{assumption}

We next discuss the statements in
Assumption~\ref{ass:detection_after_exercution}.  Statement~\ref{itm:asump 1}
ensures that the mode detection phases and stabilization phases are
alternating; in other words, the elements in $\bT^m$ and $\bT^s$ are
ordered such that $t_1^m<t_1^s<t_2^m<t_2^s<\ldots.$ This implies that the time instances $t\in\{t_i^m,\cdots,t_i^s-1\}$
are part of the $i$-th mode detection phase and that the time
instances $t\in\{t_i^s,\cdots,t_{i+1}^m-1\}$ are part of the $i$-th
stabilization phase.  From Corollary~\ref{cor:at-most-n+m}, the time
length of each mode detection phase is bounded above by $m+n$. In
simulations, we have also observed that, in general, the mode
detection phase is much shorter than this upper bound. Consequently,
as long as the unknown switching signal does not switch too
frequently, statement~\ref{itm:asump 1} holds.

Statement~\ref{itm:asump adt} is an \emph{average dwell-time} (ADT) condition, see
e.g. \cite{JPH-AM:99}, on the mode detection phase. Informally
speaking, this statement implies that, on average, the controller is
switched to mode detection phase no more than once per $\tau$ time
instances. As such, with increasing $\tau$, the frequency of switches
decreases. Therefore, this assumption again holds if the
system~\eqref{def:sys} switches relatively infrequently.

Finally, statement~\ref{itm:asump aat} is an \emph{average activation-time} (AAT)
condition, see e.g. \cite{MM-DL:12}, on the mode detection phase.  If
the condition holds, the controller dwells in mode detection phase for
at most a fraction $\eta$ of the total time. This assumption holds if
the mode detection phase is, on average, relatively short when
compared to the stabilization phase.

Note that both the ADT and AAT conditions refer to the controller and
not to the unknown switching signal~$\sigma$. Nevertheless, they are
related, as the switching of the controller will be infrequent if the
system~\eqref{def:sys} switches slowly.  This means that statements
\ref{itm:asump adt} and \ref{itm:asump aat} in Assumption~\ref{ass:detection_after_exercution} can
be considered as indirect assumptions on the switching frequency of
the unknown signal~$\sigma$.

The following result characterizes the stability properties of the
closed-loop system.

\begin{theorem}[Stability guarantee for the closed-loop
  system]\label{thm:stability_control}
  Under Assumptions~\ref{ass:ass_on_data-I},~\ref{ass:ass_on_data-II},
  and~\ref{ass:detection_after_exercution}, consider the switched
  system \eqref{def:sys} under the data-driven switching feedback
  controller described by Algorithm~\ref{alg:controller_1}. Let
  $\lambda_u,k>0$ be such that
  \begin{equation}\label{LMI_for_MD}
    \begin{bmatrix}
      P_i^{-1}&A_i&B_i\\A_i^\top&\lambda_uP_i&0\\B_i^\top&0&kI
    \end{bmatrix}\succeq 0 , \quad\forall i\in\calP ,
  \end{equation}
  and define
  \begin{equation}\label{def:mu}
    \mu:=\max_{i,j\in\calP}\Vert P_iP_j^{-1}\Vert.
  \end{equation}
  If the following holds
  \begin{equation}\label{ADT_AAT_condition}
    \left(1-\frac{\ln\lambda_u}{\ln\lambda}\right)\eta +
    \left(1-\frac{\ln\mu}{\ln\lambda}\right)\frac{1}{\tau}<1,  
  \end{equation}
  then for all initial states $x(0)\in\R^n$ and each time $t\in\N$, the
  solution of the closed-loop system satisfies
  \begin{equation}\label{ISS-estimate}
    |x(t)| \leq \sqrt{\frac{\bar\lambda\lambda b}{\underline\lambda
        a}}a^{\frac{t}{2}}|x(0)| +
    u_{\max}\sqrt{\frac{bk}{\underline\lambda(1-a)}}  ,
  \end{equation}
  where
  \begin{subequations}\label{eq:a-b}
    \begin{align}
      a&:=\lambda\left(\frac{\mu}{\lambda}\right)^{\frac{1}{\tau}}
         \left(\frac{\lambda_u}{\lambda}\right)^{\eta}
         \in(\lambda,1),
         \label{def:a} 
      \\ 
      b&:=\left(\frac{\mu}{\lambda}\right)^{\frac{1}{\tau}+N_0}
         \left(\frac{\lambda_u}{\lambda}\right)^{\eta+T_0} , \label{def:b} 
    \end{align}
  \end{subequations}
  and $\bar\lambda,\underline\lambda$ are the maximum of the largest
  eigenvalues and the minimum of the smallest eigenvalues among
  $\{P_i\}_{i \in \calP}$ respectively.
\end{theorem}

Note that the LMIs in \eqref{LMI_for_MD} always hold by
picking $\lambda_u$ and $k$ sufficiently large. With $\lambda$ given
as in Assumption~\ref{ass:ass_on_data-I}, we know that the decay rate
during the stabilization phase is faster than $\lambda$, and
$\lambda_u$ gives an upper bound on the possible growth rate during
the mode detection phase. The parameter $\mu$ corresponds to the
destabilizing effect introduced by each switching. For a given
switched system, the constants $\mu$, $\lambda$, and $\lambda_u$ are
fixed. Hence the condition~\eqref{ADT_AAT_condition} is always
satisfied if $\tau$ is sufficiently large and $\eta$ is sufficiently
small. This means as long as the system switches sufficiently
infrequently and the mode detection phases are sufficiently short,
then~\eqref{ISS-estimate} holds.

\begin{figure*}[ht]
	\tikzset{every picture/.style={scale=.65}}
	\centering
	\subfigure[]{
%
%
\definecolor{mycolor1}{rgb}{0.00000,0.44700,0.74100}%
\definecolor{mycolor2}{rgb}{0.85000,0.32500,0.09800}%
\begin{tikzpicture}

\begin{axis}[%
width=4.521in,
height=3.566in/1.3,
at={(0.758in,0.481in)},
scale only axis,
xmin=0,
xmax=100,
ymin=-1,
ymax=6,
axis background/.style={fill=white},
legend style={at={(axis cs:99,1)},anchor=north east}
]

\addplot[const plot, color=mycolor1] table[row sep=crcr] {%
1	4\\
2	4\\
3	4\\
4	4\\
5	4\\
6	4\\
7	4\\
8	4\\
9	4\\
10	4\\
11	1\\
12	1\\
13	1\\
14	1\\
15	1\\
16	1\\
17	1\\
18	1\\
19	1\\
20	1\\
21	1\\
22	1\\
23	1\\
24	1\\
25	1\\
26	1\\
27	2\\
28	2\\
29	2\\
30	2\\
31	2\\
32	2\\
33	4\\
34	4\\
35	4\\
36	4\\
37	4\\
38	4\\
39	4\\
40	2\\
41	2\\
42	2\\
43	2\\
44	2\\
45	2\\
46	5\\
47	5\\
48	5\\
49	5\\
50	5\\
51	5\\
52	5\\
53	5\\
54	5\\
55	5\\
56	2\\
57	2\\
58	2\\
59	2\\
60	2\\
61	2\\
62	2\\
63	2\\
64	4\\
65	4\\
66	4\\
67	4\\
68	4\\
69	4\\
70	4\\
71	4\\
72	3\\
73	3\\
74	3\\
75	3\\
76	3\\
77	3\\
78	3\\
79	2\\
80	2\\
81	2\\
82	2\\
83	2\\
84	2\\
85	2\\
86	5\\
87	5\\
88	5\\
89	5\\
90	5\\
91	5\\
92	5\\
93	5\\
94	2\\
95	2\\
96	2\\
97	2\\
98	2\\
99	2\\
100	2\\
101	2\\
};
\addlegendentry{$\sigma(t)$}

\addplot[const plot, color=mycolor2] table[row sep=crcr] {%
1	1\\
2	1\\
3	1\\
4	4\\
5	4\\
6	4\\
7	4\\
8	4\\
9	4\\
10	4\\
11	4\\
12	4\\
13	4\\
14	1\\
15	1\\
16	1\\
17	1\\
18	1\\
19	1\\
20	1\\
21	1\\
22	1\\
23	1\\
24	1\\
25	1\\
26	1\\
27	1\\
28	1\\
29	1\\
30	2\\
31	2\\
32	2\\
33	2\\
34	2\\
35	2\\
36	4\\
37	4\\
38	4\\
39	4\\
40	4\\
41	4\\
42	4\\
43	2\\
44	2\\
45	2\\
46	2\\
47	2\\
48	2\\
49	5\\
50	5\\
51	5\\
52	5\\
53	5\\
54	5\\
55	5\\
56	5\\
57	5\\
58	5\\
59	2\\
60	2\\
61	2\\
62	2\\
63	2\\
64	2\\
65	2\\
66	2\\
67	4\\
68	4\\
69	4\\
70	4\\
71	4\\
72	4\\
73	4\\
74	4\\
75	3\\
76	3\\
77	3\\
78	3\\
79	3\\
80	3\\
81	3\\
82	2\\
83	2\\
84	2\\
85	2\\
86	2\\
87	2\\
88	2\\
89	5\\
90	5\\
91	5\\
92	5\\
93	5\\
94	5\\
95	5\\
96	5\\
97	2\\
98	2\\
99	2\\
100	2\\
};
\addlegendentry{$\sigma_d(t)$}

\end{axis}
\end{tikzpicture}%
}
	\subfigure[]{
%
%
\definecolor{mycolor1}{rgb}{0.00000,0.44700,0.74100}%
\begin{tikzpicture}

\begin{axis}[%
width=4.521in,
height=3.566in/1.3,
at={(0.758in,0.481in)},
scale only axis,
xmin=0,
xmax=100,
ymin=-0.5,
ymax=1.5,
ytick={0,1},
yticklabels={S,MD},
axis background/.style={fill=white},
legend style={legend cell align=left, align=left, draw=white!15!black}
]
\addplot[const plot, color=mycolor1] table[row sep=crcr] {%
1	1\\
2	1\\
3	1\\
4	0\\
5	0\\
6	0\\
7	0\\
8	0\\
9	0\\
10	0\\
11	1\\
12	1\\
13	1\\
14	0\\
15	0\\
16	0\\
17	0\\
18	0\\
19	0\\
20	0\\
21	0\\
22	0\\
23	0\\
24	0\\
25	0\\
26	0\\
27	1\\
28	1\\
29	1\\
30	0\\
31	0\\
32	0\\
33	1\\
34	1\\
35	1\\
36	0\\
37	0\\
38	0\\
39	0\\
40	1\\
41	1\\
42	1\\
43	0\\
44	0\\
45	0\\
46	1\\
47	1\\
48	1\\
49	0\\
50	0\\
51	0\\
52	0\\
53	0\\
54	0\\
55	0\\
56	1\\
57	1\\
58	1\\
59	0\\
60	0\\
61	0\\
62	0\\
63	0\\
64	1\\
65	1\\
66	1\\
67	0\\
68	0\\
69	0\\
70	0\\
71	0\\
72	1\\
73	1\\
74	1\\
75	0\\
76	0\\
77	0\\
78	0\\
79	1\\
80	1\\
81	1\\
82	0\\
83	0\\
84	0\\
85	0\\
86	1\\
87	1\\
88	1\\
89	0\\
90	0\\
91	0\\
92	0\\
93	0\\
94	1\\
95	1\\
96	1\\
97	0\\
98	0\\
99	0\\
100	0\\
};

\end{axis}
\end{tikzpicture}%
}
	\subfigure[]{
		\input{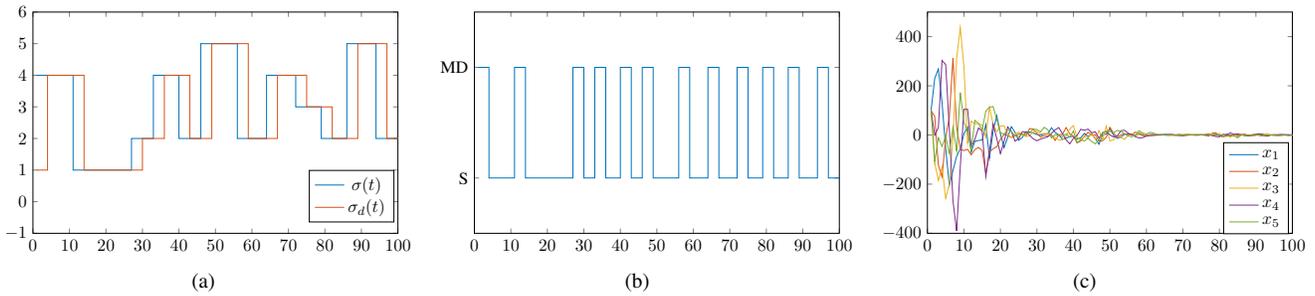}}
              \caption{The results of the example of
                Section~\ref{sec:simulation}. (a) shows the switching
                signals of the system ($\sigma(t)$) and of the
                controller ($\sigma_d(t)$), (b) shows the active phase
                of the controller, which is either mode detection or
                stabilization, and (c) shows the resulting state
                trajectories. }\label{fig:example}
\end{figure*}

\section{Simulation results}\label{sec:simulation}
We illustrate the performance of the proposed data-driven switching
feedback controller. We consider an unknown switched linear system
with parameters $n=5$, $m=3$, and $p=5$. In order to conserve space, 
we refrain from presenting the state and input matrices here. A data pair $(U^i_-,X^i)$
with $T=7$ is collected in the initialization phase for each mode
$i \in \calP = \{1,\dots,5\}$. We set $\lambda=0.8$ as the desired
decay rate for each mode. Both Assumptions~\ref{ass:ass_on_data-I} and
\ref{ass:ass_on_data-II} are satisfied on the initial data. While on
average, we let the system switch once per 8 units of time as shown by
the blue curve in Figure~\ref{fig:example}(a), the mode-to-go is set
to be random.  In the mode detection algorithm, we choose $u_{\max}=1$
as the bound on the input.  Figure~\ref{fig:example} shows the result
of implementing the data-driven switched controller on the
system. Recall that the switching signal $\sigma_d(t)$ is the
controller's determination of the active mode and the control
$u(t)=K_{\sigma_d(t)}x(t)$ is applied to the system during the
stabilization phase.  The comparison between $\sigma(t)$ and
$\sigma_d(t)$ in Figure~\ref{fig:example}(a) shows that the ``lag'' is
almost always $3$ time instants. During the mismatch between
$\sigma_d(t)$ and $\sigma(t)$, the controller switches to the mode
detection phase, cf. Figure~\ref{fig:example}(b). From this plot, one
can see that Assumption~\ref{ass:detection_after_exercution} holds in
general with parameters $\tau=8$, the average dwell time of
$\sigma(t)$, and $\eta=3/8$. Although the switching is frequent enough
such that the theoretical stability-guaranteeing inequality
\eqref{ADT_AAT_condition} is not met for this system, the resulting
solution trajectory turns out to converge towards the origin,
cf. Figure~\ref{fig:example}(c). We also observe some minor
oscillations of the solution around the origin, which are caused by
the persistent mode detection input during each mode detection phase.

\section{Conclusions}\label{sec:conclusion}
We have addressed the data-driven stabilization of switched linear
systems in scenarios where no knowledge of the system matrices of each
operating mode is available and the switching signal is also unknown.
Instead, we have access to a number of system measurements for each
mode. We have used ideas from the data informativity framework to
develop a controller that alternates between mode detection and
stabilization phases to guarantee an ISS-like property for a wide
range of unknown switching signals. Future work will address the
extension of our results to scenarios where multiple modes are
compatible with the data available in the initialization phase and to
the case of noisy measurements.

\bibliography{../bib/alias,../bib/JC,../bib/Main,../bib/Main-add,../bib/New,../bib/FB}

\begin{thebibliography}{10}
\providecommand{\url}[1]{#1}
\csname url@samestyle\endcsname
\providecommand{\newblock}{\relax}
\providecommand{\bibinfo}[2]{#2}
\providecommand{\BIBentrySTDinterwordspacing}{\spaceskip=0pt\relax}
\providecommand{\BIBentryALTinterwordstretchfactor}{4}
\providecommand{\BIBentryALTinterwordspacing}{\spaceskip=\fontdimen2\font plus
\BIBentryALTinterwordstretchfactor\fontdimen3\font minus
  \fontdimen4\font\relax}
\providecommand{\BIBforeignlanguage}[2]{{%
\expandafter\ifx\csname l@#1\endcsname\relax
\typeout{** WARNING: IEEEtran.bst: No hyphenation pattern has been}%
\typeout{** loaded for the language `#1'. Using the pattern for}%
\typeout{** the default language instead.}%
\else
\language=\csname l@#1\endcsname
\fi
#2}}
\providecommand{\BIBdecl}{\relax}
\BIBdecl

\bibitem{DL:03}
D.~Liberzon, \emph{Switching in Systems and Control}, ser. Systems \& Control:
  Foundations \& Applications.\hskip 1em plus 0.5em minus 0.4em\relax
  Birkh{\"a}user, 2003.

\bibitem{HL-PJA:09}
H.~Lin and P.~J. Antsaklis, ``Stability and stabilizability of switched linear
  systems: A survey of recent results,'' \emph{IEEE Transactions on Automatic
  Control}, vol.~54, no.~2, pp. 308--322, 2009.

\bibitem{DC-LG-YL-YW:05}
D.~Cheng, L.~Guo, Y.~Lin, and Y.~Wang, ``Stabilization of switched linear
  systems,'' \emph{IEEE Transactions on Automatic Control}, vol.~50, no.~5, pp.
  661--666, 2005.

\bibitem{YS:11switched}
Y.~Sun, ``Stabilization of switched systems with nonlinear impulse effects and
  disturbances,'' \emph{IEEE Transactions on Automatic Control}, vol.~56,
  no.~11, pp. 2739--2743, 2011.

\bibitem{MH-FF-GD:13}
M.~Hou, F.~Fu, and G.~Duan, ``Global stabilization of switched stochastic
  nonlinear systems in strict-feedback form under arbitrary switchings,''
  \emph{Automatica}, vol.~49, no.~8, pp. 2571--2575, 2013.

\bibitem{DZ-LA-JD-QZ:18}
D.~Zhai, L.~An, J.~Dong, and Q.~Zhang, ``Switched adaptive fuzzy tracking
  control for a class of switched nonlinear systems under arbitrary
  switching,'' \emph{IEEE Transactions on Fuzzy Systems}, vol.~26, no.~2, pp.
  585--597, 2018.

\bibitem{MSB:98}
M.~S. Branicky, ``Multiple {L}yapunov functions and other analysis tools for
  switched and hybrid systems,'' \emph{IEEE Transactions on Automatic Control},
  vol.~43, no.~4, pp. 475--482, 1998.

\bibitem{JG-PC:06}
J.~C. Geromel and P.~Colaneri, ``Stability and stabilization of continuous-time
  switched linear systems,'' \emph{SIAM Journal on Control and Optimization},
  vol.~45, no.~5, pp. 1915--1930, 2006.

\bibitem{PC-JG-AA:08}
P.~Colaneri, J.~C. Geromel, and A.~Astolfi, ``Stabilization of continuous-time
  switched nonlinear systems,'' \emph{Systems \& Control Letters}, vol.~57,
  no.~1, pp. 95--103, 2008.

\bibitem{HL-PJA:07}
H.~Lin and P.~J. Antsaklis, ``Switching stabilizability for continuous-time
  uncertain switched linear systems,'' \emph{IEEE Transactions on Automatic
  Control}, vol.~52, no.~4, pp. 633--646, 2007.

\bibitem{LIA-US:13}
L.~I. Allerhand and U.~Shaked, ``Robust control of linear systems via
  switching,'' \emph{IEEE Transactions on Automatic Control}, vol.~58, no.~2,
  pp. 506--512, 2013.

\bibitem{CDP-PT:20}
C.~{De Persis} and P.~Tesi, ``Formulas for data-driven control: Stabilization,
  optimality, and robustness,'' \emph{IEEE Transactions on Automatic Control},
  vol.~65, no.~3, pp. 909--924, 2020.

\bibitem{HJVW-JE-HLT-MKC:20}
H.~J. van Waarde, J.~Eising, H.~L. Trentelman, and M.~K. Camlibel, ``Data
  informativity: a new perspective on data-driven analysis and control,''
  \emph{IEEE Transactions on Automatic Control}, vol.~65, no.~11, pp.
  4753--4768, 2020.

\bibitem{CZ-MG-JZ:19}
C.~Zhang, M.~Gan, and J.~Zhao, ``Data-driven optimal control of switched linear
  autonomous systems,'' \emph{International Journal of Systems Science},
  vol.~50, no.~6, pp. 1275--1289, 2019.

\bibitem{AK:20arXiv}
A.~Kundu, ``Data-driven switching logic design for switched linear systems,''
  \emph{arXiv preprint arXiv:2003.05774}, 2020.

\bibitem{TD-MS:18}
T.~Dai and M.~Sznaier, ``Data driven robust superstable control of switched
  systems,'' \emph{IFAC-PapersOnLine}, vol.~51, no.~25, pp. 402--408, 2018, 9th
  IFAC Symposium on Robust Control Design ROCOND 2018.

\bibitem{ZW-GOB-RMJ:21arXiv}
Z.~Wang, G.~O. Berger, and R.~M. Jungers, ``Data-driven feedback stabilization
  of switched linear systems with probabilistic stability guarantees,''
  \emph{arXiv preprint arXiv:2103.10823}, 2021.

\bibitem{MR-CDP-PT:22}
M.~Rotulo, C.~{De Persis}, and P.~Tesi, ``Online learning of data-driven
  controllers for unknown switched linear systems,'' \emph{Automatica}, vol.
  145, p. 110519, 2022.

\bibitem{HJVW-MKC-JE-HLT:22}
H.~J. van Waarde, M.~K. Camlibel, J.~Eising, and H.~L. Trentelman, ``Quadratic
  matrix inequalities with applications to data-based control,'' \emph{arXiv
  preprint arXiv:2203.12959}, 2022.

\bibitem{HJW-MKC:21arXiv}
H.~J. van Waarde and M.~K. Camlibel, ``A matrix {Finsler's} lemma with
  applications to data-driven control,'' \emph{arXiv preprint
  arXiv:2103.13461}, 2021.

\bibitem{HJVW:21}
H.~J. van Waarde, ``Beyond persistent excitation: Online experiment design for
  data-driven modeling and control,'' \emph{IEEE Control Systems Letters},
  vol.~6, pp. 319--324, 2022.

\bibitem{JPH-AM:99}
J.~P. {Hespanha} and A.~S. {Morse}, ``Stability of switched systems with
  average dwell-time,'' in \emph{{IEEE} Conf.\ on Decision and Control},
  Shangai, China, Dec. 1999, pp. 2655--2660.

\bibitem{MM-DL:12}
M.~A. M{\"u}ller and D.~Liberzon, ``Input/output-to-state stability and
  state-norm estimators for switched nonlinear systems,'' \emph{Automatica},
  vol.~48, no.~9, pp. 2029--2039, 2012.

\end{thebibliography}


\bibliographystyle{IEEEtran}

\end{document}